\documentclass[11pt,a4paper]{article}

\usepackage{epsf,epsfig,amsfonts,amsgen,amsmath,amstext,amsbsy,amsopn,amsthm}
\usepackage{amsmath,times,mathptmx}
\usepackage{amsfonts,amsthm,amssymb}
\usepackage{amsfonts}
\usepackage{graphics}
\usepackage{latexsym,bm}
\usepackage{amsfonts,amsthm,amssymb,bbding}
\usepackage{indentfirst}
\usepackage{graphicx}
\usepackage{color}
\usepackage[colorlinks=true,anchorcolor=blue,filecolor=blue,linkcolor=blue,urlcolor=blue,citecolor=blue]{hyperref}
\usepackage{float}
\usepackage{tikz}
\evensidemargin=\oddsidemargin\topmargin=-1.5cm

\newtheorem{thm}{Theorem}[section]

\newtheorem{proc}{Procedure}[section]
\newtheorem{lem}{Lemma}[section]

\newtheorem{claim}{Claim}[section]
\newtheorem{definition}{Definition}[section]

\addtocounter{section}{0}

\begin{document}
\title{Non-$r$-partite graphs without complete split subgraphs\footnote{Supported by the Major Natural
Science Research Project of Universities in Anhui Province (Nos. 2023AH051589),
the Domestic Academic Visit Project of Universities in Anhui Province (No. JNFX2023064), the National Natural Science Foundation of China (No. 12201408), and the Natural Science Research Project of Chuzhou University (No. 2022XJYB17).}}
\author{ {\bf Bing Wang $^{a,b}$}, {\bf Wenwen Chen $^a$},
{\bf Ping Zhang$^{c}$}\thanks{Corresponding author: czcww@chzu.edu.cn (W. Chen).}
\\
\small $^{a}$ School of Mathematics and Finance, Chuzhou University,\\
\small Chuzhou, Anhui 239012, China\\
\small $^{b}$ School of Mathematical Sciences, East China Normal University, \\
\small Shanghai 200241, China \\
\small $^{c}$ College of Science, University of Shanghai for Science and Technology,\\
\small  Shanghai 200093, China\\
}

\date{}
\maketitle
{\flushleft\large\bf Abstract}  The classical Simonovits' chromatic critical edge theorem shows that for sufficiently large $n$, if $H$ is an edge-color-critical graph with $\chi(H)=p+1\ge 3$, then the Tur\'an graph $T_{n,p}$ is the unique extremal graph with respect to ${\rm ex}(n,H)$. Denote by ${\rm EX}_{r+1}(n,H)$ and ${\rm SPEX}_{r+1}(n,H)$ the family of $n$-vertex $H$-free non-$r$-partite graphs with the maximum size and with the spectral radius, respectively. Li and Peng [SIAM J. Discrete Math. 37 (2023)  2462--2485] characterized the unique graph in $\mathrm{SPEX}_{r+1}(n,K_{r+1})$ for $r\geq 2$ and showed that $\mathrm{SPEX}_{r+1}(n,K_{r+1})\subseteq \mathrm{EX}_{r+1}(n,K_{r+1})$. It is interesting to study the
extremal or spectral extremal problems for color-critical graph $H$ in non-$r$-partite graphs.
For $p\geq 2$ and $q\geq 1$, we call the graph $B_{p,q}:=K_p\nabla qK_1$ a complete split graph (or generalized book graph). In this note, we determine the unique spectral extremal graph in $\mathrm{SPEX}_{r+1}(n,B_{p,q})$ and show that  $\mathrm{SPEX}_{r+1}(n,B_{p,q})\subseteq \mathrm{EX}_{r+1}(n,B_{p,q})$ for sufficiently large $n$.
\begin{flushleft}
\textbf{Keywords:} Extremal graph; Spectral radius; Complete split graph
\end{flushleft}
\textbf{AMS Classification:} 05C35; 05C50

\section{Introduction}

In this paper we will consider some classical problems in extremal
combinatorics. For general extremal problems, one aims to maximize some parameter among a class of graphs avoiding some other
structures. For a given family of graphs $\mathcal{F}$, we say that $G$ is \emph{$\mathcal{F}$-free} if $G$ does not contain any member of $\mathcal{F}$ as a subgraph. If $\mathcal{F}$ contains exactly a member $F$, then we say that $G$ is $F$-free for short. For a graph $F$, the \emph{Tur\'an number} of $F$, denoted by $\text{ex}(n,F)$, is the maximum number of edges possible in an $n$-vertex $F$-free graph. An $F$-free graph with $n$ vertices and $\text{ex}(n,F)$ edges is called an \emph{extremal graph} for $F$. For convenience, let $\text{EX}(n,F)$ denote the family of $n$-vertex $F$-free graphs with the maximum number of edges. Besides, for an integer $k\ge1$, let $[k]=\{1,2,\ldots,k\}$.

A well-known theorem of Mantel \cite{Mantel07} shows that if $G$ is an $n$-vertex $K_3$-free graph,
then $e(G)\leq e(K_{\lceil{\frac{n}2}\rceil,\lfloor{\frac{n}2}\rfloor})=\lfloor{\frac{n^2}4}\rfloor$.
Mantel's theorem has many applications and generalizations in the literature.
For example, Erd\H{o}s gave the following stability form of Mantel's theorem (see \cite{Bondy2008}, Page 306):
If $G$ is an $n$-vertex non-bipartite $K_3$-free graph,
then $e(G)\leq \lfloor{\frac{(n-1)^2}4}\rfloor+1$.
The upper bound is best possible and the extremal graph is not unique. Let $T_{n,r}$ be the $n$-vertex complete $r$-partite graph and all its part sizes as equal as possible.
The famous Tur\'{a}n's theorem \cite{Turan} shows that for any integer $r$, $\text{ex}(n, K_{r+1})\le e(T_{n,r})$, and equality holds when extremal graph is isomorphic to Tur\'an graph $T_{n,r}$. The classical Tur\'an's problem and its variations have been paid much attention and many influential results have been obtained, we refer to a survey \cite{Nikiforov11} for more details.

\begin{definition}
Assume that $T_1,\dots,T_p$ are color classes of $T_{n-1, p}$ with $n_i=|T_i|$ for $i\in [p]$.
Moreover, we assume that $\lfloor \frac{n-1}{p}\rfloor=n_1\leq n_2\leq \cdots \leq n_p=\lceil\frac{n-1}{p}\rceil$,
$u_1\in T_1$ and  $u_2\in T_2$.
Let $Y_r(n)$ be the graph obtained from $T_{n-1, p}$ by removing the edge $u_1u_2$, then adding all edges between a new vertex $u_0$ and $(\cup_{i=3}^{p}T_i)\cup \{u_1,u_2\}$, see Figure \ref{fig-1.1}.
\end{definition}

In 1981, Brouwer gave the following statement on Tur\'an's theorem.

\begin{thm}\label{thm0}\emph{(Brouwer, \cite{Brouwer1981})}
Let $n\ge 2r+1$ and $G$ be an $n$-vertex $K_{r+1}$-free non-$r$-partite graph. Then $e(G)\le e(T_r(n))-\lfloor{\frac{n}r}\rfloor+1$.
\end{thm}
Similar as stability form of Mantel's theorem, the upper bound of Theorem \ref{thm0} is best possible and there are many extremal graphs attaining the bound. For example, the graph $Y_r(n)$ is one of extremal graph of Theorem \ref{thm0}.

\begin{figure}
\centering
\begin{tikzpicture}[scale=0.8, x=1.00mm, y=1.00mm, inner xsep=0pt, inner ysep=0pt, outer xsep=0pt, outer ysep=0pt]
\path[line width=0mm] (86.97,28.00) rectangle +(75.03,54.00);
\definecolor{L}{rgb}{0,0,0}
\path[line width=0.30mm, draw=L] (120.00,75.00) ellipse (28.00mm and 4.00mm);
\path[line width=0.30mm, draw=L] (120.00,35.00) ellipse (28.00mm and 4.00mm);
\path[line width=0.30mm, draw=L] (120.00,55.00) ellipse (28.00mm and 4.00mm);
\draw(149,73) node[anchor=base west]{\fontsize{10}{17.07}\selectfont $T_1$};
\draw(149,53) node[anchor=base west]{\fontsize{10}{17.07}\selectfont $T_{2}$};
\draw(149,33) node[anchor=base west]{\fontsize{10}{17.07}\selectfont $T_p$};

\definecolor{F}{rgb}{0,0,0}
\path[line width=0.60mm, draw=L] (130,75) -- (130,55);
\path[line width=0.60mm, draw=L] (140,35) -- (140,55);
\path[line width=0.60mm, draw=L] (135,75) -- (135,35);

\path[line width=0.30mm, draw=L] (70,65) -- (95,75);
\path[line width=0.30mm, draw=L] (70,65) -- (95,55);
\path[line width=0.60mm, draw=L] (70,65) -- (95,35);

%\path[line width=0.60mm, draw=L] (115,75) -- (95,35);
%\path[line width=0.60mm, draw=L] (120,75) -- (115,55);

\path[line width=0.30mm, draw=L, dash pattern=on 2.00mm off 1.00mm] (95,75) -- (95,55);

\path[line width=0.30mm, draw=L, fill=F] (95.00,75.00) circle (1.00mm);
\path[line width=0.30mm, draw=L, fill=F] (95.00,55.00) circle (1.00mm);
%\path[line width=0.30mm, draw=L, fill=F] (95,35) circle (1.00mm);
\path[line width=0.30mm, draw=L, fill=F] (70,65) circle (1.00mm);

\draw(64.00,64.00) node[anchor=base west]{\fontsize{10}{17.07}\selectfont $u_0$};
\draw(97.00,54.00) node[anchor=base west]{\fontsize{10}{17.07}\selectfont $u_{2}$};
\draw(97.00,74.00) node[anchor=base west]{\fontsize{10}{17.07}\selectfont $u_1$};

\path[line width=0.10mm, draw=L, fill=F] (120,42) circle (0.50mm);
\path[line width=0.10mm, draw=L, fill=F] (120,45) circle (0.50mm);
\path[line width=0.10mm, draw=L, fill=F] (120,48) circle (0.50mm);

\end{tikzpicture}%
\caption{The graph $Y_r(n)$.}{\label{fig-1.1}}
\end{figure}
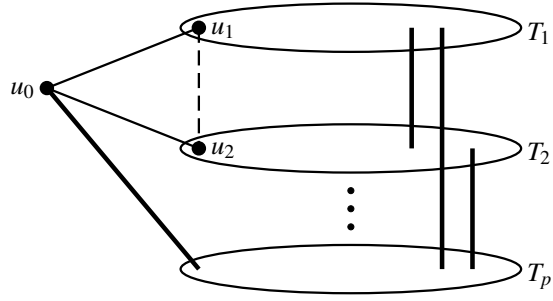

Given an $n$-vertex graph $G$, let $A(G)$ and $\rho(G)$ denote the \emph{adjacency matrix} and the \emph{spectral radius} of $G$. Let $\text{SPEX}(n,F)$ denote the family of $n$-vertex $F$-free graphs with the maximum spectral radius. In 1970, Nosal \cite{Nosal1970} showed that if $G$ is an $n$-vertex $K_3$-free graph with size $m$ then $\rho(G)\le \sqrt{m}$, with equality holds if and only if $G$ is a complete bipartite graph. Combining with Mantel's theorem, the following spectral
Mantel theorem holds: If $G$ is a $K_3$-free graph, then $\rho(G)\le \sqrt{m}\le \sqrt{\lfloor{\frac{n^2}4}\rfloor}=\rho(K_{\lceil{\frac{n}2}\rceil,\lfloor{\frac{n}2}\rfloor})$. In 1986, Wilf \cite{Wilf1986} showed the spectral version of Tur\'an theorem as follows: Let $G$ be an $n$-vertex $K_{r+1}$-free graph, then $\rho(G) \leq \big(1-\frac1r\big)n$.
Nikiforov \cite{Nikiforov5} showed that if $G$ is an $n$-vertex $K_{r+1}$-free graph, then $\rho(G)\le \rho({T_{n, r}})$, with equality holds if and only if $G\cong T_{n,r}$. In 2021, Lin, Ning, and Wu \cite{LIN1} generalized Nosal's result for nonbipartite triangle-free graphs. Li and Peng considered the maximum spectral radius of non-bipartite graphs forbidding short odd cycles
\cite{LIY}. Based on the spectral Zykov symmetrization \cite{Zykov1949}, Li and Peng extended the result of Lin, Ning and Wu to non-$r$-partite $K_{r+1}$-free graphs as follows.

\begin{thm}\label{thm1.0}\emph{(Li and Peng, \cite{Li2023})}
Let $G$ be an $n$-vertex non-$r$-partite $K_{r+1}$-free graph. Then
$\rho(G) \leq \rho(Y_r(n))$, with equality holds if and only if $G\cong Y_r(n)$.
\end{thm}

In 2022, Cioab\u{a}, Desai and Tait \cite{Cioaba2022} conjectured that for a given graph $F$, if $G\in \text{EX}(n,F)$ and $G$ is obtained from Tur\'an graph by adding $O(1)$ edges, then $\text{SPEX}(n,F)\subseteq \text{EX}(n,F)$ for sufficiently large $n$. This conjecture has been confirmed for some special cases of $F$ (for example, complete graphs \cite{Guiduli-1996, Nikiforov5},  friendship graphs \cite{CFTZ2020,ZHAI2022}, intersecting cliques and intersecting odd cycles \cite{DKL2022,Li}) and completely solved by Wang, Kang and Xue \cite{WANG} in 2023.

For an integer $t\ge 1$, we denote $B_t=K_2\nabla tK_1$ the \emph{book graph} with its book size $t$. For two integers $p\ge 2,q\ge 1$, denoted by $B_{p,q}=K_p\nabla qK_1$ the \emph{generalized book graph} of book size $q$. Specially, $B_{p,q}\cong B_q$ for $p=2$ and $B_{p,q}\cong K_{p+1}$ for $q=1$. The generalized book graph $B_{p,q}$ is also called a \emph{complete split graph} (denoted by $S_{p+q,q}$, respectively), which is usual as an extremal graph in spectral extremal theory, see \cite{Nikiforov1} and elsewhere.

Let $\mathcal{G}(n,p,q)$ be the set of graphs with the largest number of edges among all $n$-vertex non-$p$-partite $B_{p,q}$-free graphs. In the paper, we obtain the following results on non-$r$-partite $B_{p,q}$-free graphs respect to Tur\'an-type extremal and spectral Tur\'an-type extremal version, respectively.

\begin{thm}\label{thm1.1}
Let $p\geq 3$, $q\geq 1$ and $n$ be sufficiently large.
Then $Y_r(n)\in \mathcal{G}(n,p,1)\subseteq \mathcal{G}(n,p,q)$.
\end{thm}

\begin{thm}\label{thm1.2}
Let $p\geq 3$, $q\geq 1$ and $n$ be sufficiently large. If $G$ is an $n$-vertex $B_{p,q}$-free graph then $\rho(G)\le \rho(Y_r(n))$,  with equality holds if and only if $G\cong Y_r(n)$.
\end{thm}

Note that $B_{p,q}$ contains $K_{p+1}$ as a subgraph. So Theorem \ref{thm1.2} means that Theorem \ref{thm1.0} holds for sufficiently large $n$.

\section{Preliminaries}
In this section, we first present some lemmas that will be used in the following proof. For a given graph $H$, the minimum integer $k$ for which $H$ is $k$-colourable is called its \emph{chromatic number}, and denoted by $\chi(H)$.
A graph $H$ is called an \emph{edge-color-critical graph}, if there exists an edge $e\in E(H)$
such that $\chi(H-e)<\chi(H)$. The following result, due to Simonovits, plays an important role in extremal graph theory.

 \begin{lem}\label{lem2.1}\emph{(Simonovits, \cite{S1966})}
For $r\geq 2$ and sufficiently large $n$, if $H$ is an edge-color-critical graph with $\chi(H)=r+1$,
then the Tur\'an graph $T_{n,r}$ is the unique extremal graph with respect to ${\rm ex}(n,H)$.
 \end{lem}

Erd\H{o}s \cite{Erdos-1967,Erdos-1968}
and Simonovits \cite{S1966} posed the classical stability theorem as follows.

\begin{lem} \label{lem2.2}
\emph{(Erd\H{o}s and  Simonovits, \cite{Erdos-1967,Erdos-1968,S1966})}
Let $H$ be a given graph with $\chi(H)=r+1$ and $G$ be an $n$-vertex $H$-free graph.
For every $\varepsilon>0$ and $r\geq 2$,
there exist a constant $\delta>0$ and an integer $n_0$ such that
if $n\geq n_0$ and $e(G)\geq (\frac{r-1}{r}-\delta)\frac{n^2}{2}$,
then $G$ differs from $T_{n,r}$ at most $\varepsilon n^2$ edges.
\end{lem}

The following lemma is easy to check by double counting.
\begin{lem} \label{counting}
If $S_1,\ldots,S_k$ are $k$ finite sets, then $|S_1\cap\ldots\cap S_k|\ge
\sum\limits_{i=1}^k|S_i|-(k-1)|\bigcup\limits_{i=1}^k S_i|$.
\end{lem}

\section{Proof of Theorem \ref{thm1.1}}\label{section4}

In this section, we provide the proof of Theorem \ref{thm1.1}.
Let $G$ be a graph in $\mathcal{G}(n,p,q)$ with its minimum degree $\delta(G)=\min_{F\in \mathcal{G}(n,p,q)}\delta(F)$.
We always assume that $n$ is large enough.
Since $Y_r(n)$ is $K_{p+1}$-free and $B_{p,q}$ contains a subgraph isomorphic to $K_{p+1}$, $Y_r(n)$ is $B_{p,q}$-free.
By the choice of $G$, we have
\begin{align}\label{align-2a}
e(G)\geq e(Y_r(n))=e(T_{n,p})-\Big\lfloor\frac{n}{p}\Big\rfloor+1.
\end{align}

If $\chi(G)\leq p-1$, then by Tur\'an's theorem, we have $e(G)\leq e(T_{n,p-1})<e(T_{n,p})-\big\lfloor\frac{n}{p}\big\rfloor+1$,
which contradicts \eqref{align-2a}.
Furthermore, since $G$ is non-$p$-partite, we have $\chi(G)\geq p+1$.

For a subset $U$ of $V(G)$, denote by $e(U)$ the number of edges with all its ends in $U$ and $G[U]$ the induced subgraph by $U$. For two disjoint subset $U_1,U_2$ of $V(G)$, denote by $e(U_1,U_2)$ the number of edges with one end in $U_1$ and the other in $U_2$. For convenience, we define
\begin{align}\label{align-1}
 0<\eta<\frac{1}{7p^5q}, ~0\le \varepsilon<\eta^3,
\end{align}
which will be used frequently in the sequel.

\begin{lem}\label{lem2.3}
$G$ admits a vertex partition $V(G)=\bigcup_{i=1}^{p}V_i$ such that
$\sum_{i=1}^{p}e(V_i)\leq \eta^3 n^2$.
Furthermore, for any partition $\bigcup_{i=1}^{p}V_i$ with $\sum_{i=1}^{p}e(V_i)\leq \eta^3 n^2$, we have $\big||V_i|-\frac{n}{p}\big|\leq\eta n$ for each $i\in [p]$.
\end{lem}

\begin{proof}
By \eqref{align-2a} and Lemma \ref{lem2.2}, $G$ can be obtained from $T_{n,p}$ by adding and deleting at most $\varepsilon n^2$ edges. It follows that
there exists a vertex partition $V(G)=\bigcup_{i=1}^{p} V_i^*$ such that $\big\lfloor\frac{n}{p}\big\rfloor\leq |V^*_i|\leq \big\lceil\frac{n}{p}\big\rceil$ and
$\sum_{i=1}^{p}e(V^*_i)\leq \eta^3 n^2$ for each $i\in [p]$.

Now we will show that for any vertex partition $V(G)=\bigcup_{i=1}^{p}V_i$
satisfying that $\sum\limits_{i=1}^{p}e(V_i)\leq \eta^3 n^2$, we have
$\big||V_i|-\frac{n}{p}\big|\leq\eta n$ for each $i\in [p]$.
Set $\gamma=\max\{\big||V_i|-\frac{n}{p}\big|,~i\in [p]\}$. Without loss of generality, say $\gamma=||V_1|-\frac{n}{p}|$.
From the Cauchy-Schwarz inequality, we have $\sum_{i=2}^{p}|V_i|^2\geq \frac{1}{p-1}(\sum_{i=2}^{p}|V_i|)^2$.
Then
   $$2\sum_{2\leq i<j\leq p}|V_i||V_j|=\Big(\sum_{i=2}^{p}|V_i|\Big)^2-\sum_{i=2}^{p}|V_i|^2
   \leq \frac{p-2}{p-1}(n-|V_1|)^2.$$
Consequently,
\begin{eqnarray*}
e(G)&=& \sum_{1\leq i<j\leq p}e(V_i,V_j)+\sum_{i=1}^{p}e(V_i)\\
&\leq& |V_1|(n-|V_1|)+\sum_{2\leq i<j\leq p}|V_i||V_j|+\eta^3 n^2\\
&\leq& |V_1|(n-|V_1|)+\frac{p-2}{2(p-1)}(n-|V_1|)^2+\eta^3 n^2\\
&=& -\frac{p}{2(p-1)}\gamma^2+\frac{p-1}{2p}n^2+\eta^3 n^2.
\end{eqnarray*}

On the other hand, by Lemma \ref{lem2.2},
\begin{align}\label{align-1.1}
   e(G)\geq e(T_{n,p})-\eta^3 n^2>\big(\frac{p-1}{2p}-2\eta^3\big) n^2.
\end{align}
Combining the above two inequalities gives $\frac{p}{2(p-1)}\gamma^2<3\eta^3 n^2$.
Thus, $\gamma<\sqrt{\frac{6(p-1)\eta^3}{p}n^2}<\eta n$.
\end{proof}

For a vertex $v$ and a vertex subset $X$ of $G$, let $N_X(v)$ be the set of neighbors of $v$ in $X$ and $d_X(v)=|N_X(v)|$. Specially, if $X=V(G)$, then we write $N_G(v)$ ($d_G(v)$, respectively) for short.

By Lemma \ref{lem2.3}, we assume that $\bigcup_{i=1}^{p}V_{i}$ is a vertex partition of $G$ such that
$\sum_{i=1}^{p}e(V_i)$ attains the minimum. We now define two subsets of vertex of $G$ and then characterize their properties in following lemmas.
 Let $S=\{v\in V(G)~|~d_G(v)\leq \big(\frac{p-1}{p}-5\eta\big)n\}$, $W_i=\{v\in V_i~|~d_{V_i}(v)\geq 2\eta n\}$ and $W=\bigcup_{i=1}^{p}W_i$. Then we can obtain the following statements.

\begin{lem}\label{lem2.4}
$|S|\leq \eta n$.
\end{lem}

\begin{proof}
Suppose to the contrary that $|S|>\eta n$.
Then we choose a subset $S'\subseteq S$ such that $|S'|=\lfloor\eta n\rfloor$.
Let $n'=|V(G-S')|=n-\lfloor\eta n\rfloor$. By (\ref{align-1.1}), we have
\begin{align*}
 e(G-S')\geq  e(G)-\sum_{v\in S'}d_G(v)
 \geq \Big(\frac{p-1}{2p}(1-2\eta)+5\eta^2-\eta^3\Big)n^2>\frac{p-1}{2p}(n-\lfloor\eta n\rfloor)^2\geq e(T_{n',p}).
\end{align*}
Since $B_{p,q}$ is an edge-color-critical graph with $\chi(B_{p,q})=p+1$, by Lemma \ref{lem2.1}, $G-S'$ contains a copy of $B_{p,q}$, a contradiction.
\end{proof}

\begin{lem}\label{lem2.5}
$|W|\leq \frac{1}{2}\eta n$.
\end{lem}

\begin{proof}
For every $i\in [p]$,
\begin{center}
  $2e(V_i)=\sum\limits_{v\in V_i}d_{V_i}(v)\geq
\sum\limits_{v\in W_i}d_{V_i}(v)\geq 2\eta n|W_i|.$
\end{center}
Combining this with Lemma \ref{lem2.3}, we have
\begin{center}
  $\eta^3 n^2\geq \sum\limits_{i=1}^{p}e(V_i)\geq \sum\limits_{i=1}^{p}\eta n|W_i|=\eta n|W|.$
\end{center}
This yields that $|W|\leq \eta^2 n\leq \frac{1}{2}\eta n$.
\end{proof}

Set $\overline{V}_i=V_i\setminus (S\cup W)$ for every $i\in [p]$.
Then we have the following statements.

\begin{lem}\label{lem2.6}
Let $i_0\in [p]$ and $j_0\in [pq]$ be two integers. Then \\
(a) $d_{V_{i_0}}(u)\geq\big(\frac{1}{p^2}-2p\eta\big)n$ for any $u\in \bigcup_{k\in [p]\setminus \{i_0\}}(W_{k}\setminus S)$ (if the vertex $u$ exists); \\
(b) $d_{V_{i_0}}(u)\geq\big(\frac{1}{p}-4p\eta\big)n$ for any $u\in \bigcup_{k\in [p]\setminus \{i_0\}}\overline{V}_{k}$;\\
(c) if $u\in \bigcup_{k\in [p]\setminus \{i_0\}}(W_{k}\setminus S)$ and
$\{u_1,\dots,u_{j_0}\}\subseteq \bigcup_{k\in [p]\setminus \{i_0\}}\overline{V}_k$, then there are at least $q$ vertices in $\overline{V}_{i_0}$ adjacent to $u_1,\dots,u_{j_0}$ and $u$.

\end{lem}

\begin{proof}
(a) Suppose that there is a vertex $u\in W_{k_0}\setminus S$ for some ${k_0}\in [p]\setminus \{i_0\}$.
Since $\bigcup_{i=1}^{p}V_{i}$ is a partition of $V(G)$ such that $\sum_{i=1}^{p}e(V_i)$ attains the minimum (that is, $\sum_{1\leq i<j\leq p}e(V_{i},V_{j})$ is the maximum),
$d_{V_{{k_0}}}(u)\leq \frac{1}{p}d_{G}(u)$.
By Lemma \ref{lem2.3}, $|V_{k}|\leq \big(\frac{1}{p}+\eta\big)n$ for any $k\in [p]$.
Together with $d_G(u)> \big(\frac{p-1}{p}-5\eta\big)n$ (as $u\notin S$), we have
\begin{align*}
  d_{V_{i_0}}(u)&= d_G(u)-d_{V_{k_0}}(u)-\sum\limits_{k\in [p]\setminus\{i_0,k_0\}}d_{V_{k}}(u)\\
  &\geq \frac{p-1}{p}d_G(u)- (p-2)\Big(\frac{1}{p}+\eta\Big)n\\
  &\geq \Big(\frac{1}{p^2}-2p\eta\Big)n.
\end{align*}

(b)
Assume that $u\in \overline{V}_{k_0}$ for some $k_0\in [p]\setminus \{i_0\}$.
 Since $u\notin S\cup W_{k_0}$, we have $d_{V_{k_0}}(u)<2\eta n$ and $d_G(u)> \big(\frac{p-1}{p}-5\eta\big)n$.
Note that $d_{V_k}(u)\leq |V_{k}|\leq \big(\frac{1}{p}+\eta\big)n$ for any $k\in [p]\setminus\{i_0,k_0\}$. Then
\begin{align*}
  d_{V_{i_0}}(u)
    &= d_G(u)-d_{V_{k_0}}(u)-\sum\limits_{k\in [p]\setminus\{i_0,k_0\}}d_{V_{k}}(u)\\
    &\geq \Big(\frac{p-1}{p}-5\eta\Big)n-2\eta n- (p-2)\Big(\frac{1}{p}+\eta\Big)n\\
    &\geq \Big(\frac{1}{p}-4p\eta\Big)n.
\end{align*}

(c) By (a) and  (b), we have $d_{V_{i_0}}(u)\geq \big(\frac{1}{p^2}-2p\eta\big)n$ and $d_{V_{i_0}}(u_k)\geq\big(\frac{1}{p}-4p\eta\big)n$ for any $k\in [j_0]$.
Then by Lemma \ref{counting}, we have
\begin{align*}
\Big|N_{V_{i_0}}(u)\cap\Big(\bigcap_{k=1}^{j_0}N_{V_{i_0}}(u_k)\Big)\Big|
           &\geq  |N_{V_{i_0}}(u)|+\sum_{k=1}^{j_0}|N_{V_{i_0}}(u_k)|-j_0|V_{i_0}|\\
            &\geq  \Big(\frac{1}{p^2}-2p\eta\Big)n+j_0\Big(\frac{1}{p}-4p\eta\Big)n
                       -j_0\Big(\frac{1}{p}+\eta\Big)n\\
            &\geq  \Big(\frac{1}{p^2}-7pj_0\eta\Big)n\\
           &\geq  |S\cup W|+q,
\end{align*}
where the last inequality holds as $j_0\leq pq$, $\eta<\frac{1}{7p^5q}$, $|S\cup W|\leq |S|+|W|\leq \frac{3\eta n}{2}$, and $n$ is sufficiently large.
Then, there exist at least $q$ vertices in $\overline{V}_{i_0}$ adjacent to $u_1,\dots,u_{j_0}$ and $u$.
\end{proof}

\begin{lem}\label{lem2.7}
For each $i\in [p]$, $G[\overline{V}_i]$ is empty and $W_i\subseteq S$.
\end{lem}

\begin{proof}
We first show that $G[\overline{V}_1]$ is empty. Otherwise, say, $G[\overline{V}_1]$ contains an edge $u_{1,1}u_{1,2}$. Then we set $\widehat{V}_1=\{u_{1,1},u_{1,2}\}$.
Recursively applying Lemma \ref{lem2.6} (c), we can obtain a sequence: $\widehat{V}_2,\dots,\widehat{V}_p$
such that for any $k\in [p]\setminus\{1\}$,
$\widehat{V}_k=\{u_{k,1},u_{k,2},\dots,u_{k,q}\}\subseteq \overline{V}_k$ and all vertices in $\widehat{V}_k$ are adjacent to all vertices in $\bigcup_{i=1}^{k-1}\widehat{V}_i$. Then $G[\{u_{1,1},u_{1,2},u_{2,1},\dots,u_{p-1,1}\}\cup \widehat{V}_p]$ contains a subgraph isomorphic to $B_{p,q}$, a contradiction. Hence, $G[\overline{V}_1]$ is empty. By symmetry, $G[\overline{V}_i]$ is empty for each $i\in [p]$.

Now we will show that $W_1\subseteq S$. Suppose to the contrary that there exists a vertex $u_{1,1}\in W_1\setminus S$. Note that $d_{V_1}(u_{1,1})\geq 2\eta n$ and $|S\cup W|\leq \frac{3}{2}\eta n$. By Lemmas \ref{lem2.4} and \ref{lem2.5},
we have $d_{\overline{V}_1}(u_{1,1})\geq d_{V_1}(u_{1,1})-|S\cup W|\geq \frac12\eta n$.
Thus, $G[\overline{V}_1]$ is not empty, which contradicts the statement above. Similarly, we have $W_i\subseteq S$ for each $i\in [p]$. This completes the proof.
\end{proof}

Let $u^*\in V_{i_0} $ for some $i_0\in [p]$ be a vertex with the maximum degree of $G$.
By \eqref{align-1.1}, we have $d_G(u^*)\geq \frac{2e(G)}{n}>(\frac{p-1}{p}-\eta)n.$
This implies that $u^*\notin S$. Hence $u^*\notin W$ by Lemma \ref{lem2.7}.
It follows that $d_{\overline{V}_{i_0}}(u^*)\leq d_{V_{i_0}}(u^*)< 2\eta n.$
Combining this with Lemma \ref{lem2.4}, we obtain
\begin{align*}
d_G(u^*)\leq |S|+|N_{V_{i_0}}(u^*)|+\sum_{i\in [p]\setminus\{i_0\}}|\overline{V}_{i}|
< 3\eta n+\sum_{i\in [p]\setminus\{i_0\}}|\overline{V}_{i}|.
\end{align*}
It implies that
\begin{align}\label{align-2}
\sum_{i\in [p]\setminus\{i_0\}}|\overline{V}_{i}|> d_G(u^*)-3\eta n> \Big(\frac{p-1}{p}-4\eta\Big)n.
\end{align}

\begin{lem}\label{lem2.9}
$S\neq \varnothing$ and $\chi(G-\{u\})=p$ for any $u\in S$.
\end{lem}

\begin{proof}
Suppose first to the contrary that $S= \varnothing$.
By Lemma \ref{lem2.7}, we know that for each $i\in [p]$, $W_i=\varnothing$. Therefore, $G[V_i]$ is empty for each $i\in [p]$.
This means $\chi(G)\leq p$, a contradiction.

We now claim that $\chi(G-\{u\})=p$ for any $u\in S$.
Otherwise,
there is a vertex $u_0\in S\cap V_{i_0}$ for some $i_0\in [p]$ such that $\chi(G-\{u_0\})\geq p+1$  (since $\chi(G)\geq p+1$).
Let $G'$ be the graph obtained from $G$ by deleting all edges incident to $u_0$
and adding all possible edges between $(\bigcup_{i\in [p]\setminus\{i_0\}}\overline{V}_i)\setminus \{u_0\}$ and $u_0$.
We will show that $G'$ is also $B_{p,q}$-free. Otherwise, suppose that $G'$ contains a subgraph $H$ isomorphic to $B_{p,q}$.
From the construction of $G'$, $u_0\in V(H)$. Let $u_1,u_2,\dots,u_c\in \bigcup_{i\in [p]\setminus\{i_0\}}\overline{V}_i$ be $c$ neighbors of $u_0$ in $H$, where $c\leq |V(H)|-1\leq pq$.
By Lemma \ref{lem2.6} (c), we can select a new vertex $u\in \overline{V}_{i_0}\setminus V(H)$ adjacent to all $u_1,u_2,\dots,u_c$.
This implies that $G[(V(H)\setminus \{u_0\})\cup\{u\}]$ also contains a copy of $B_{p,q}$, a contradiction.

On the other hand, by \eqref{align-2} and the definition of $S$, we obtain
\begin{align*}\label{align-2}
e(G')-e(G)\geq \bigg|\Big(\bigcup_{i\in [p]\setminus\{i_0\}}\overline{V}_i\Big)\setminus \{u_0\}\bigg|-d_G(u_0)
>\Big(\frac{p-1}{p}-4\eta\Big)n-\Big(\frac{p-1}{p}-5\eta\Big)n>0.
\end{align*}
Since $G-\{u_0\}$ is a subgraph of $G'$, $\chi(G')\geq \chi(G-\{u_0\})\geq p+1$. But it contradicts the choice of $G$.
The proof is complete.
\end{proof}

Let $u_0$ be a vertex of $G$ with $d_G(u_0)=\min\{d_G(u): u\in V(G)\}$.
Since $S\neq \varnothing$, $u_0\in S$.
By Lemma \ref{lem2.9}, we know that $\chi(G-\{u_0\})=p$.
Let $U_1,U_2,\dots,U_p$ be $p$ color classes of $G-\{u_0\}$.
For each $i\in [p]$, set $U_{i}'=N_G(u_0)\cap U_i$.
It is clear that each $U_{i}'$ is non-empty
(Otherwise, $U_1,\dots,U_{i-1},U_i\cup \{u_0\},U_{i+1},\dots,U_p$ are $p$ color classes of $G$. It follows that
 $\chi(G)\leq p$, a contradiction).
Note that $U_1\cup \{u_0\},U_2,\dots,U_p$ consist of a vertex partition of $G$ with
$e(U_1\cup \{u_0\})+\sum_{i=2}^{p}e(U_i)<N_G(u_0)\leq \big(\frac{p-1}{p}-5\eta\big)n<\eta^3 n^2$ for sufficiently large $n$. By Lemma \ref{lem2.3}, $\big||U_i|-\frac{n}{p}\big|\leq\eta n$ for each $i\in [p]$.

\begin{proc}\label{proc1}
Set $G_1:=G$.
We begin with $i=1$ and proceed the following steps to generate a sequence of graphs $G_1,G_2,\dots, G_\ell$.

 \noindent
 ($i$) In the $i$-th step, we let $U_{s,i}=N_{G_i}(u_0)\cap U_s$ and define the following three types of $U_{s,i}$ for each $s\in [p]$:\\
$\bullet$ Type ($i,A$): if $N_{G_i}(u)\setminus \{u_0\}=\cup_{s'\in [p]\setminus \{s\}}U_s$ for each vertex $u\in U_{s,i}$;\\
$\bullet$ Type ($i,B$): if $U_{s,i}$ contains exactly one element, say $u$, and
 $N_{G_i}(u)\setminus \{u_0\}\subsetneq \cup_{s'\in [p]\setminus \{s\}}U_s$;\\
$\bullet$ Type ($i,C$): otherwise.

\noindent
($ii$) If there exists an integer $s\in [p]$ such that $U_{s,i}$ is of type ($i,C$),
then there is a vertex $u_{i}$ in $U_{s,i}$ such that $N_{G_i}(u_{i})\setminus \{u_0\}\subsetneq \cup_{s'\in [p]\setminus \{s\}}U_s$.
Let $G_{i+1}$ be the graph obtained from $G_i$ by deleting the edge $u_0u_{i}$
and adding all possible edges between $\big(\cup_{s'\in [p]\setminus \{s\}}U_{s'}\big)\setminus N_{G_i}(u_i)$ and $u_{i}$.
Set $U_{s,i+1}=U_{s,i}\setminus\{u_{i}\}$ and  $U_{s',i+1}=U_{s',i}$ for each $s'\in [p]\setminus \{s\}$.
Go to the $(i+1)$-step.

\noindent
($iii$) If for each $s\in [p]$, $U_{s,i}$ is not of Type ($i,C$),
then we stop the procedure and set $G^*=G_i$.
\end{proc}

Clearly, the procedure will end in finite steps.
We denote by $G_{\ell}$ the resulting graph.
For any $\{i,j\}\subset [p]$,
let $G_{i,j}$ be the graph obtained from the complete $p$-partite graph with $p$ color classes $U_1,\dots,U_p$ of $G-\{u_0\}$  (denote as above) by deleting an edge $u_iu_j$ (where $u_i\in U_i$ and  $u_j\in U_j$), then adding all edges between $u_0$ and $(\cup_{s\in [p]\setminus \{i,j\}}U_s)\cup \{u_i,u_j\}$.

\begin{lem}\label{lemma3.7C}
$G_{\ell}\cong  G_{i_0,j_0}$ for some $\{i_0,j_0\}\subseteq [p]$.
Moreover, $e(G)=e(G_{i_0,j_0})$.
\end{lem}

\begin{proof}
We partition the proof into three claims as follows.

\begin{claim}\label{claim3.1FF}
$e(G)\geq e(G_{i,j})$ for any $\{i,j\}\subset [p]$.
\end{claim}

\begin{proof}
Note that $\big||U_i|-\frac{n}{p}\big|\leq\eta n$ for each $i\in [p]$. Then
 $Y_r(\frac{n}{2})\subseteq G_{i,j}\subseteq Y_r(2n)$.
Note that $Y_r(2n)$ is $K_{p+1}$-free and $\chi\left(Y_r\big(\frac{n}{2}\big)\right)=\chi(Y_r(2n))=p+1$.
It follows that $\chi(G_{i,j})=p+1$ and $G_{i,j}$ is $B_{p,q}$-free.
By the choice of $G$, $e(G)\geq e(G_{i,j})$ for any $\{i,j\}\subset [p]$.
\end{proof}

\begin{claim}\label{claim3.1B}
$G_{\ell}$ is $B_{p,q}$-free.
\end{claim}

\begin{proof}
Clearly, $\chi(G_{\ell}-\{u_0\})=p$. Suppose to the contrary
that $G_{\ell}$ contains a subgraph $H$ isomorphic to $B_{p,q}$.
Since $\chi(H)=p+1$ and $H-\{u_0\}\subseteq G_{\ell}-\{u_0\}$ , we have $\chi(H-\{u_0\})=p$ and $u_0\in V(H)$.
Then, $u_0$ is adjacent to all vertices of $H$ as $q\geq 2$.
This means that $H\subseteq G_{\ell}[N_{G_{\ell}}(u_0)\cup \{u_0\}]\subseteq G[N_{G}(u_0)\cup \{u_0\}]$.
However, this is impossible as $G$ is $B_{p,q}$-free.
\end{proof}

 Recall that $U_{s,\ell}=N_{G_\ell}(u_0)\cap U_i$ for each $s\in [p]$. Then the following statement holds.
\begin{claim}\label{claim3.1C}
There exists some $t\in [p]$ such that $U_{t,\ell}$ is of Type ($\ell,A$).
\end{claim}

\begin{proof}
Without loss of generality, assume that $|U_{1,\ell}|=\max_{s\in [p]}|U_{s,\ell}|$.
Clearly,
\begin{equation*}
e(G_\ell)\leq e(K_{|U_1|,\dots,|U_p|})+\sum_{s\in [p]}|U_{s,\ell}|,
\end{equation*}
and
\begin{equation*}
e(G_{i,j})=e(K_{|U_1|,\dots,|U_p|})-1+\sum_{s\in [p]\setminus \{i,j\}}|U_s|.
\end{equation*}
Combining  Claim \ref{claim3.1FF} and $\big||U_s|-\frac{n}{p}\big|\leq\eta n$ for each $s\in [p]$,
we obtain
$$p|U_{1,\ell}|\geq \sum_{s\in [p]}|U_{s,\ell}|
\geq \sum_{s\in [p]\setminus \{i,j\}}|U_s|-1\geq \frac{n}{2p},$$
which implies that $|U_{1,\ell}|\geq \frac{n}{2pq}\geq \max\{q,2\}\ge 2$.
So, $U_{1,\ell}$ is of Type ($\ell,A$).
\end{proof}

For two vertex subsets $U_1$ and $U_2$, let $G[U_1,U_2]$ be the subgraph of $G$ with its vertex set $U_1\cup U_2$ and edge set $\{uv|u\in U_1,v\in U_2\}$.
\begin{claim}\label{claim3.2B}
$G_{\ell}[U_{i_0,\ell},U_{j_0,\ell}]$ is not a complete bipartite subgraph of $G_{\ell}$ for some $\{i_0,j_0\}\subset [p]$.
\end{claim}

\begin{proof}
Otherwise, $G_{\ell}[\{u_0\}\cup (\cup_{s\in [p]}U_{s,\ell})]$ is a complete $(p+1)$-partite graph. By Claim \ref{claim3.1C}, we have $|U_{t,\ell}|\geq q$ for some $t\in [p]$.
Then, $G_{\ell}[\{u_0\}\cup (\cup_{s\in [p]}U_{s,\ell})]$ contains a copy of $B_{p,q}$.
On the other hand, since $G[\{u_0\}\cup (\cup_{s\in [p]}U_{s,\ell})]=G_{\ell}[\{u_0\}\cup (\cup_{s\in [p]}U_{s,\ell})]$, we can see that $G_{\ell}$ also contains a copy of $B_{p,q}$,
which contradicts Claim \ref{claim3.1B}.
\end{proof}
By the definition of $G_\ell$, we can see that
\begin{equation}\label{equ-1}
e(G_{\ell})\geq e(G_{\ell-1})\geq \cdots \geq e(G_{1})=e(G).
\end{equation}
By Claim \ref{claim3.2B}, $G_{\ell}[U_{i_0,\ell},U_{j_0,\ell}]$ is not a complete bipartite subgraph
 for some $\{i_0,j_0\}\subseteq [p]$.
Considering the definition of $G_\ell$,
we can see that both $U_{i_0,\ell}$ and $U_{j_0,\ell}$ are of Type ($\ell,B$).
Assume that $U_{i_0,\ell}=\{u_{i_0}\}$ and $U_{j_0,\ell}=\{u_{j_0}\}$.
Clearly, $u_{i_0}u_{j_0}\notin E(G_\ell)$.
Thus, $G_\ell$ is a subgraph of $G_{i_0,j_0}$.
This, together with \eqref{equ-1} and Claim \ref{claim3.1FF}, implies that $G_\ell\cong G_{i_0,j_0}$ and $e(G)=e(G_{i_0,j_0})$, as desired.
\end{proof}

Note that $G_{i_0,j_0}$ is an $n$-vertex $K_{p+1}$-free non-$p$-partite graph.
By the definition of $Y_r(n)$, we have $e(G_{i_0,j_0})\leq e(Y_r(n))$.
On the other hand, we can see that $Y_r(n)$ is $K_{p+1}$-free and  non-$p$-partite.
Since $K_{p+1}\subseteq B_{p,q}$, we know that $Y_r(n)$ is $B_{p,q}$-free.
By the choice of $G$ and Lemma \ref{lemma3.7C}, we have $e(G_{i_0,j_0})=e(G)\geq e(Y_r(n)).$
Therefore, $e(G)=e(G_{i_0,j_0})=e(Y_r(n))$.

For any $H\in \mathcal{G}(n,p,1)$,
since $H$ is $K_{p+1}$-free,  it follows that $H$ is also $B_{p,q}$-free.
Combining $e(G)=e(Y_r(n))=e(H)$ gives $H\in \mathcal{G}(n,p,q)$.
This leads to that $\mathcal{G}(n,p,1)\subseteq \mathcal{G}(n,p,q)$.
This completes the proof of Theorem \ref{thm1.1}.

\section{Proof of Theorem \ref{thm1.2}}\label{section5}
In this section, we first give some necessary lemmas as follows.

\begin{lem}\emph{(Wu, Xiao and Hong, \cite{Wu2005})}\label{Wu}
 Let $G$ be a connected graph and $\text{\bf x}$ be the Perron vector of $G$ with the coordinate $x_v$ of $\text{\bf x}$ corresponding to the vertex $v$ of $G$. Assume that there are two vertices $u,v$ of $G$ satisfying that $x_u \ge x_v$ and $v$ has $s\ge 1$ private neighbors of $v$ (that is, not adjacent to $u$), say $v_1,\ldots,v_s$. Let $G'=G-\{vv_i:i\in [s]\}+\{uv_i:i\in [s]\}$. Then $\rho(G')>\rho(G)$.
 \end{lem}

 In 2009, Nikiforov \cite{Nikiforov4} gave the following spectral stability theorem.

\begin{thm}\label{thm3.1}\emph{(Nikiforov, \cite{Nikiforov4})}
For $r\ge 2$, $0<\varepsilon<2^{-36}r^{-24}$ and $\frac{1}{\ln n}<c<r^{-8(r+21)(r+1)}$,
If $G$ is an $n$-vertex graph with $\rho(G)>(1-\frac1r-\varepsilon)n$, then either $G$ differs from $T_{n,r}$ in fewer than $(\varepsilon^{\frac{1}{4}}+c^{\frac{1}{8r+8}})n^2$ edges, or $G$ contains a $K_{r+1}(\lfloor c\ln n\rfloor, \dots,\lfloor c\ln n\rfloor,\lceil n^{1-\sqrt{c}}\rceil)$.
\end{thm}

Based on Theorem \ref{thm3.1}, Desai et al. \cite{DKL2022} derived the following spectral version stability result.

\begin{lem} \label{lem3.5}\emph{(Desai et al., \cite{DKL2022})}
Let $F$ be a graph with $\chi(F)=r+1\ge 3$.
For every $\varepsilon>0$, there exist $\delta>0$ and $n_0$ such that
if $G$ is an $n$-vertex $F$-free graph with $\rho(G)\geq (1-\frac1r-\delta)n$,
then $G$ can be obtained from $T_{n,r}$ by adding and deleting at most $\varepsilon n^2$ edges.
\end{lem}

Based on Nikiforov's result above, Zhai and Lin gave the following spectral version of the color-critical theorem.

\begin{lem} \label{lem3.4}\emph{(Nikiforov, Zhai and Lin, \cite{Nikiforov2009, ZHAI2023})}
Let $r\geq 2$ and $H$ be a color-critical graph with $\chi(H)=r+1$.
Then there exists an $n_0\geq e^{|V(H)|r^{(2r+9)(r+1)}}$ such that ${\rm SPEX}(n,H)=\{T_{n,r}\}$ for $n\ge n_0$.
\end{lem}

\begin{proof}[\bf{Proof of Theorem \ref{thm1.2}}] For $p\geq 3$, $q\geq 1$, and $n$ sufficiently large,
let $G$ be an $n$-vertex $B_{p,q}$-free graph attaining the maximal spectral radius.
Based on Li and Peng's result in Theorem \ref{thm1.0}, to prove Theorem \ref{thm1.2},
it suffices to show that $G$ is $K_{r+1}$-free.
Obviously, $G$ is connected (since adding any edge $e$ between two components increases the spectral radius and does not form a $B_{p,q}$ in $G+e$).
Note that $Y_r(n)$ is $B_{p,q}$-free. Then
\begin{equation}\label{equ-2}
\rho(G)\ge \rho(Y_r(n))\ge \frac{2e(Y_r(n))}{n}\ge \Big(\frac{p-1}p-\eta \Big)n,
\end{equation}
where $\eta$ is defined in (\ref{align-1}).

By Lemma \ref{lem3.5}, $G$ can be obtained from $T_{n,r}$
by adding and deleting at most $\varepsilon n^2$ edges. It follows that all of Lemmas \ref{lem2.3}-\ref{lem2.7} hold for the spectral extremal graph $G$. We still say that $G$ has a vertex partition $V(G)=\bigcup_{i=1}^{p}V_i$ such that
$\sum_{i=1}^{p}e(V_i)$ attains the minimum.

Let $S=\{v\in V(G)~|~d_G(v)\leq \big(\frac{p-1}{p}-5\eta\big)n\}$, $W_i=\{v\in V_i~|~d_{V_i}(v)\geq 2\eta n\}$, $W=\bigcup_{i=1}^{p}W_i$ and $\overline{V}_i=V_i\setminus (S\cup W)$ for every $i\in [p]$. We relabel Lemmas \ref{lem2.4} and \ref{lem2.7} as follows for completeness and then give a spectral version of Lemma \ref{lem2.9}.

\begin{lem}\label{lem3.62}
$|S|\leq \eta n$.
\end{lem}

\begin{lem}\label{lem3.7}
For each $i\in [p]$, $G[\overline{V}_i]$ is empty and $W_i\subseteq S$.
\end{lem}

Let $\text{\bf x}=(x_1,\ldots, x_n)^\text{T}$ be the unit positive eigenvector corresponding to $\rho(G)$
 with the coordinate $x_v$ corresponding to the vertex $v$ of $G$. Set $x_{u^*}=\max\{x_v:v\in V(G)\}$ and
assume that $u^*\in V_{i_1}$ for some $i_1\in [p]$.

\begin{lem}\label{lem3.8EF}
$\sum\limits_{i\in [p]\setminus\{i_1\}}\sum\limits_{v\in {\overline{V}_i}\setminus\{u_0\}}x_v> (\frac{p-1}{p}- 2\eta)n$.
%(ii) $S\neq \varnothing$ and $\chi(G-\{u\})=p$ for any $u\in S$.
\end{lem}

\begin{proof}
Clearly, $\rho(G)x_{u^*}=\sum_{v\in N_{G}(u^*)}x_v\leq d_G(u^*)x_{u^*}$.
Then $d_G(u^*)\geq \rho(G)\geq \big(\frac{p-1}{p}-\eta\big)n$. It implies that $u^*\notin S$.
By Lemma \ref{lem3.7}, we get
$u^*\in \overline{V}_{i_1}$, $d_{\overline{V}_{i_1}}(u^*)=0$ and $d_{V_{i_1}}(u^*)\le|S|\le \eta n$. Then
\begin{align}\label{eq4.1}
\rho(G)x_{u^*}&=\!\!\sum_{v\in N_{S}(u^*)}x_v+\sum\limits_{i\in [p]\setminus \{i_1\}}\sum_{v\in N_{\overline{V}_i}(u^*)}x_v
\leq \eta n x_{u^*}+\sum\limits_{i\in [p]\setminus \{i_1\}}\sum_{v\in {\overline{V}_i}}x_v.
\end{align}
Combining (\ref{equ-2}) gives
\begin{align*}
\sum\limits_{i\in [p]\setminus \{i_1\}}\sum_{v\in {\overline{V}_i}\setminus\{u_0\}}x_v
\geq (\rho(G)- \eta n -1)x_{u^*}
> \Big(\frac{p-1}{p}- 2\eta\Big)n,
\end{align*}
as desired.
\end{proof}

\begin{lem}\label{lem3.62DF}
Let $u'\in V(G)$ with $x_{u'}=\min\{x_v:v\in V(G)\}$.
Then $\chi(G-\{u'\})=p$.
\end{lem}

\begin{proof}
Otherwise, $\chi(G-\{u'\})\geq p+1$.
For any $u_0\in S$, $d_G(u_0)< (\frac{p-1}{p}- 5\eta)n$.
Combining this with Lemma \ref{lem3.8EF}, we have
\begin{align}\label{eq4.2CD}
\sum\limits_{v\in N_G(u')}x_v=\rho(G)x_{u'}\leq \rho(G)x_{u_0}\leq d_G(u_0)x_{u^*}
<\sum\limits_{i\in [p]\setminus \{i_1\}}\sum_{v\in {\overline{V}_i}\setminus\{u_0\}}x_v.
\end{align}
Let $G'$ be the graph obtained from $G$ by deleting all edges incident to $u'$ and
adding all possible edges between $(\bigcup_{i\in [p]\setminus\{i_0\}}\overline{V}_i)\setminus \{u'\}$ and $u'$.
Similar as the proof of Lemma \ref{lem2.9}, we can obtain that $G'$ is also $B_{p,q}$-free.
By \eqref{eq4.2CD},
\begin{align*}
  \rho(G')-\rho(G) \geq  \mathbf{x}^{\mathsf{T}}\big(A(G')-A(G)\big)\mathbf{x}
                  = 2x_{u'}\Big(\sum\limits_{i\in [p]\setminus \{i_1\}}\sum_{v\in {\overline{V}_i}\setminus\{u'\}}x_v-\sum\limits_{v\in N_G(u')}x_v\Big)>0.
\end{align*}
Since $G-\{u'\}$ is a subgraph of $G'$, it follows that $\chi(G')\geq \chi(G-\{u'\})\geq p+1$.
Thus, $G'$ is an $n$-vertex $B_{p,q}$-free non-$p$-partite graph,
which contradicts the choice of $G$.
\end{proof}

Let $U_1,U_2,\dots,U_p$ be $p$ color classes of $G-\{u'\}$ and $U_{i}'=N_G(u')\cap U_i$.
Similar as the discussion after Lemma \ref{lem2.9}, each $U_{i}'$ is non-empty and $\big||U_i|-\frac{n}{p}\big|\leq\eta n$ for each $i\in [p]$. We still perform Procedure \ref{proc1} and denote $G_{\ell}$ the resulting graph. Recall the definition of $G_{i,j}$, we have the following statements.

\begin{lem}\label{lemma4.9}
$G=G_\ell$ and $G\cong G_{i_0,j_0}$ for some $\{i_0,j_0\}\subseteq [p]$.
\end{lem}

\begin{proof}
Note that $G=G_1$.
Suppose that $\ell\geq 2$.
For each $i\in [\ell-1]$,
there exists a vertex $v_i\in \big(\cup_{s'\in [p]\setminus \{s\}}U_{s'}\big)\setminus N_{G_i}(u_i)$
such that $u_iv_i\in E(G_{i+1})\setminus E(G_i)$.
Furthermore, $u_iv_i\in E(G_{\ell})\setminus E(G)$.
By the definition of $u'$, $x_{v_i}\ge x_{u'}$. Then
\begin{align}\label{eq4.4}
  \rho(G_{\ell})-\rho(G) \geq  \mathbf{x}^{\mathsf{T}}\big(A(G')-A(G)\big)\mathbf{x}
                  \geq 2 x_{u_i}\sum\limits_{i\in [\ell-1]}(x_{v_i}-x_{u'})
                  \geq 0.
\end{align}
If $\rho(G_{\ell})=\rho(G)$, then $\mathbf{x}$ is also a positive eigenvector of $G_{\ell}$.
Thus, $$\rho(G)x_{u'}=\rho(G_{\ell})x_{u'}=\sum_{v\in N_{G_{\ell}}(u')}x_v<\sum_{v\in N_{G}(u')}x_v,$$
which contradicts that $\rho(G)x_{u'}=\sum_{v\in N_{G}(u')}x_v$.
So, $\ell=1$ and $G\cong G_\ell$.

By a similar discussion as the proof of Lemma \ref{lemma3.7C},
$G$ is a subgraph of $G_{i_0,j_0}$ for some $\{i_0,j_0\}\subseteq [p]$.
Now we prove that $G\cong G_{i_0,j_0}$.
Otherwise, $G$ is a proper subgraph of $G_{i_0,j_0}$.
Then, $\rho(G)<\rho(G_{i_0,j_0})$.
On the other hand,
since $G_{i_0,j_0}$ is an $n$-vertex $B_{p,q}$-free non-$p$-partite graph,
by the choice of $G$ we get $\rho(G_{i_0,j_0})\leq \rho(G)$, a contradiction.
Thus, $G\cong G_{i_0,j_0}$.
\end{proof}

Now we ready to complete the proof of Theorem \ref{thm1.2}. Note that $G_{i_0,j_0}$ is $K_{p+1}$-free and non-$p$-partite. Then
by Theorem \ref{thm1.0}, $\rho(G_{i_0,j_0})\leq \rho(Y_r(n))$,
with equality holds if and only if $G_{i_0,j_0}\cong Y_r(n)$.
Since $K_{p+1}\subseteq B_{p,q}$, we know that $Y_r(n)$ is $B_{p,q}$-free and  non-$p$-partite.
By the choice of $G$,
we have $\rho(G)\geq \rho(Y_r(n))$. Combining these with Lemma \ref{lemma4.9},
we can obtain that $G\cong G_{i_0,j_0}\cong Y_r(n)$.
This completes the proof.
\end{proof}


\begin{thebibliography}{99}
\setlength{\itemsep}{0pt}

\bibitem {Bondy2008}
J. A. Bondy, U. S. R. Murty, Graph Theory,  \emph{Grad. Texts Math. 244}, Springer, New
York, 2008.

\bibitem {Brouwer1981}
A. E. Brouwer, Some Lotto Numbers from an extension of Tur\'an's Theorem, \emph{Report ZW152,
Stichting Mathematisch Centrum}, Amsterdam, 1981.

\bibitem{Cioaba2022}
S. Cioab\u{a}, D.N. Desai, M. Tait, The spectral radius of graphs with no odd wheels,  \emph{European
J. Combin.} \textbf{99} (2022), Paper No. 103420, 19 pp.

\bibitem {CFTZ2020}
S. Cioab\u{a}, L.H. Feng, M. Tait, X.-D. Zhang, The maximum spectral radius of graphs without friendship subgraphs, \emph{Electron. J. Combin.} \textbf{27} (4) (2020), Paper No. 4.22, 19 pp.

\bibitem {DKL2022}
D.N. Desai, L.Y. Kang, Y.T. Li, Z.Y. Ni, M. Tait, J. Wang, Spectral extremal graphs for intersecting cliques, \emph{Linear Algebra Appl.} \textbf{644} (2022), 234--258.

\bibitem {Erdos-1967}
P. Erd\H{o}s,  Some recent results on extremal problems in graph theory (Results),
In: Theory of Graphs (International Symposium Rome, 1966), Gordon and
Breach, New York, Dunod, Paris, 1966, pp. 117--123.

\bibitem {Erdos-1968}
P. Erd\H{o}s, On some new inequalities concerning extremal properties of
graphs, In: Theory of Graphs (Proceedings of the Colloquium, Tihany, 1966),
Academic Press, New York, 1968, pp. 77--81.

%\bibitem {E1962}
%P. Erd\H{o}s, \"{U}ber ein Extremalproblem in der Graphentheorie (German), \emph{Arch. Math. (Basel)} \textbf{13} (1962), 122--127.

%\bibitem {EFGG1995}
%P. Erd\H{o}s, Z. F\"{u}redi, R.J. Gould, D.S. Gunderson, Extremal graphs for intersecting triangles, \emph{J. Combin. Theory Ser. B} \textbf{64} (1) (1995), 89--100.

%\bibitem {EG1959}
%P. Erd\H{o}s, T. Gallai, On maximal paths and circuits of graphs, \emph{Acta Math. Acad. Sci. Hungar.} \textbf{10} (1959), 337--356.

%\bibitem {FS1975}
%R.J. Faudree, R.H. Schelp, Path Ramsey numbers in multicolorings, \emph{J. Combin. Theory Ser. B} \textbf{19} (2) (1975), 150--160.
%
%\bibitem {FG2015}
%Z. F\"{u}redi, D.S. Gunderson, Extremal numbers for odd cycles, \emph{Combin. Probab. Comput.} \textbf{24} (4) (2015), 641--645.

%\bibitem {FS2013}
%Z. F\"{u}redi, M. Simonovits,
%The history of degenerate (bipartite) extremal graph problems, Erd\"{o}s centennial, 169--264,
%Bolyai Soc. Math. Stud., 25, J\'{a}nos Bolyai Math. Soc., Budapest, 2013.
%
%\bibitem {G1111}
%R. Glebov, Extremal graphs for clique-paths, arXiv:1111.7029v1.

\bibitem {Guiduli-1996}
B.D. Guiduli, Spectral extrema for graphs, Ph.D. Thesis, 105 pages,
University of Chicago, December 1996.

\bibitem {Li2023}
Y.T. Li, Y.J. Peng, Refinement of spectral Tur\'an's Theorem, \emph{SIAM J. Discrete Math.} \text{37} (4) (2023), 2462--2485.

\bibitem {Li}
Y.T. Li, Y.J. Peng, The spectral radius of graphs with no intersecting odd cycles, \emph{Discrete Math.} \textbf{345} (8) (2022), Paper No. 112907, 16 pp.

\bibitem {LIY}
Y.T. Li, Y.J. Peng, The maximum spectral radius of non-bipartite graphs forbidding short odd cycles, \emph{Electron. J. Combin.} \textbf{29} (4) (2022), Paper No. 4.2, 27 pp.

\bibitem {LIN1}
H.Q. Lin, B. Ning, B. Wu, Eigenvalues and triangles in graphs, \emph{Combin. Probab. Comput.} \textbf{30} (2) (2021), 258--270.

%\bibitem {LIN2}
%H.Q. Lin, B. Ning, A complete solution to the Cvetkovi\'{c}-Rowlinson conjecture, \emph{J. Graph Theory} \textbf{97} (3) (2021), 441--450.

%\bibitem {L2013}
%H. Liu, Extremal graphs for blow-ups of cycles and trees, \emph{Electron. J. Combin.} \textbf{20} (1) (2013), Paper 65, 16 pp.

\bibitem{Mantel07}
W. Mantel, Problem 28, Solution by H. Gouwentak, W. Mantel, J. Teixeira de Mattes, F.
Schuh and W. A. Wythoff , Wiskundige Opgaven, 10 (1907), pp. 60--61.

%\bibitem {M1968}
%J.W. Moon, On independent complete subgraphs in a graph, \emph{Can. J. Math.} \textbf{20} (1968), 95--102.

%\bibitem {NWK2023}
%Z.Y. Ni, J. Wang, L.Y. Kang, Spectral extremal graphs for disjoint cliques, \emph{Electron. J. Combin.} \textbf{30} (1) (2023), Paper No. 1.20, 16 pp.

\bibitem {Nikiforov5}
V. Nikiforov, Bounds on graph eigenvalues II, \emph{Linear Algebra Appl.} \textbf{427} (2-3) (2007), 183--189.

\bibitem{Nikiforov11}
V. Nikiforov, Some new results in extremal graph theory, in: Surveys in Combinatorics 2011, in: London Math. Soc. Lecture Note Ser., vol. 392, Cambridge Univ. Press, Cambridge, 2011, pp. 141--181.

\bibitem {Nikiforov4}
V. Nikiforov, Stability for large forbidden subgraphs, \emph{J. Graph Theory} \textbf{62} (4) (2009), 362--368.

\bibitem {Nikiforov2009}
V. Nikiforov, Spectral saturation: inverting the spectral Tur\'an theorem,  \emph{Electron. J. Combin.}
\textbf{16} (1) (2009), Research Paper 33, 9 pp.

\bibitem {Nikiforov1}
V. Nikiforov, The spectral radius of graphs without paths and cycles of specified length, \emph{Linear Algebra Appl.}
\textbf{432} (9) (2010), 2243--2256.

\bibitem{Nosal1970}
E. Nosal, Eigenvalues of Graphs, Master's thesis, University of Calgary, 1970.

\bibitem {S1966}
M. Simonovits, A method for solving extremal problems in graph theory, stability problems, in: Theory of Graphs (Proc. Colloq., Tihany, 1966), Academic Press, New York,
1968, pp. 279--319.

%\bibitem {Simonovits1974}
%M. Simonovits, Extremal graph problems with symmetrical extremal graphs, Additional chromatic conditions, \emph{Discrete Math.} \textbf{7} (1974), 349--376.

%\bibitem {TAIT2}
%M. Tait, The Colin de Verdiere parameter, excluded minors, and the spectral radius, \emph{J. Combin. Theory Ser. A} \textbf{166} (2019), 42--58.
%
%\bibitem {TAIT1}
%M. Tait, J. Tobin, Three conjectures in extremal spectral graph theory, \emph{J. Combin. Theory Ser. B} \textbf{126} (2017), 137--161.

\bibitem {Turan}
 P. Tur\'{a}n, On an extremal problem in graph theory, \emph{Mat. Fiz. Lapok} \textbf{48} (1941), 436--452.

\bibitem {WANG}
J. Wang, L.Y. Kang, Y.S. Xue, On a conjecture of spectral extremal problems, \emph{J. Combin. Theory Ser. B} \textbf{159} (2023), 20--41.

%\bibitem {WNK2024}
%J. Wang, Z.Y. Ni, L.Y. Kang, Y.-Z. Fan,
%Spectral extremal graphs for edge blow-up of star forests,
%\emph{Discrete Math.} \textbf{347} (10) (2024), Paper No. 114141.

\bibitem{Wilf1986}
H. Wilf, Spectral bounds for the clique and indendence numbers of graphs , \emph{J. Combin. Theory
Ser. B} \textbf{40} (1986), pp. 113--117.

\bibitem{Wu2005}
B.F. Wu, E.L. Xiao, Y. Hong, The spectral radius of trees on $k$ pendant vertices, \emph{Linear
Algebra Appl.} \textbf{395} (2005), 343--349.

%\bibitem {Y2022}
%L.-T. Yuan, Extremal graphs for edge blow-up of graphs, \emph{J. Combin. Theory Ser. B} \textbf{152} (2022), 379--398.
%
%\bibitem {Yuan-2021}
%L.-T. Yuan, Extremal graphs for odd wheels,
%\emph{J. Graph Theory} \textbf{98} (2021), no. 4, 691--707.

%\bibitem {ZHAI}
%M.Q. Zhai, H.Q. Lin, Spectral extrema of $K_{s,t}$-minor free graphs--on a conjecture of M. Tait, \emph{J. Combin. Theory Ser. B} \textbf{157} (2022), 184--215.

\bibitem {ZHAI2023}
M.Q. Zhai, H.Q. Lin, A strengthening of the spectral chromatic critical edge theorem: books
and theta graphs, \emph{J. Graph Theory} \textbf{102} (3) (2023), 502--520.

\bibitem {ZHAI2022}
M.Q. Zhai, R.F. Liu, J. Xue, A unique characterization of spectral extrema for friendship
graphs,  \emph{Electron. J. Combin.} \textbf{29} (2022), no. 3, Paper No. 3.32, 17 pp.

\bibitem{Zykov1949}
A. A. Zykov, On some properties of linear complexes, \emph{Mat. Sbornik N.S.} \textbf{24} (1949),
pp. 163--188 (in Russian).

\end{thebibliography}
\end{document}